\documentclass[11pt]{amsart}
\usepackage{amsmath,amssymb,amsthm,mathrsfs}
\usepackage[all]{xy}

\newtheorem{Def}{Definition}[section]
\newtheorem{Thm}[Def]{Theorem}
\newtheorem{Prop}[Def]{Proposition}
\newtheorem{Cor}[Def]{Corollary}

\newtheorem{Ex}[Def]{Example}

\newcommand{\C}{\mathbb{C}}

\newcommand{\R}{\mathbb{R}}
\newcommand{\Z}{\mathbb{Z}}
\newcommand{\N}{\mathbb{N}}
\newcommand{\Q}{\mathbb{Q}}

\newcommand{\rk}{\mathop{\mathrm{rank}}\nolimits}

\begin{document}
\title{Trilinear forms and Chern classes of Calabi--Yau Threefolds}
\author{Atsushi Kanazawa \and P.M.H. Wilson}
\date{}
\maketitle

\begin{abstract}
Let $X$ be a Calabi--Yau threefold and $\mu$ the symmetric trilinear form on the second cohomology group $H^{2}(X,\Z)$ defined by the cup product. 
We investigate the interplay between the Chern classes $c_{2}(X)$, $c_{3}(X)$ and the trilinear form $\mu$, and demonstrate some numerical relations between them. 
When the cubic form $\mu(x,x,x)$ has a linear factor over $\R$, some properties of the linear form and the residual quadratic form are also obtained. 
\end{abstract}

\section{Introduction}
This paper is concerned with the interplay of the symmetric trilinear form $\mu$ on the second cohomology group $H^{2}(X,\Z)$ 
and the Chern classes $c_{2}(X)$, $c_{3}(X)$ of a Calabi--Yau threefold $X$.  
It is an open problem whether or not the number of topological types of Calabi--Yau threefolds is bounded and 
the original motivation of this work was to investigate topological types of Calabi--Yau threefolds via the trilinear form $\mu$ on $H^{2}(X,\Z)$. 
The role that the trilinear form $\mu$ plays in the geography of $6$-manifolds is indeed prominent as 
C.T.C. Wall proved the following celebrated theorem by using surgery methods and homotopy information associated with these surgeries.
\begin{Thm}[C.T.C. Wall \cite{Wall}] \label{Wall}
Diffeomorphism classes of simply-connected, spin, oriented, closed $6$-manifolds $X$
with torsion-free cohomology correspond bijectively to isomorphism classes of systems of invariants consisting of 
\begin{enumerate}
\item free Abelian groups  $H^{2}(X,\Z)$ and $H^{3}(X,\Z)$, 
\item a symmetric trilinear from $\mu:H^{2}(X,\Z)^{\otimes 3}\rightarrow H^{6}(X,\Z)\cong \Z$ defined by $\mu(x,y,z)=x\cup y \cup z$, 
\item a linear map $p_{1}:H^{2}(X,\Z)\rightarrow  H^{6}(X,\Z)\cong \Z$ defined by $p_{1}(x)=p_{1}(X)\cup x$,
where $p_{1}(X) \in H^{4}(X,\Z)$ is the first Pontrjagin class of $X$, 
\end{enumerate}
subject to: for any $x,y \in H = H^{2}(X,\Z)$, 
$$
\mu(x,x,y)+\mu(x,y,y)\equiv 0 \pmod{2},\ \
4\mu(x,x,x) -p_{1}(x)\equiv 0\pmod{24}.
$$
The isomorphism $H^{6}(X,\Z)\cong \Z$ above is given by pairing the cohomology class with the fundamental class $[X]$ with natural orientation. 
\end{Thm}
At present the classification of trilinear forms, which is as difficult as that of diffeomorphism classes of $6$-manifolds, is unknown.  
In the light of the essential role of the K3 lattice in the study of K3 surfaces, we would like to propose the following question:
\textit{what kind of trilinear forms $\mu$ occur on Calabi--Yau threefolds?} 
The quantized version of the trilinear forms, known as Gromov--Witten invariants or A-model Yukawa couplings, are also of interest to both mathematicians and physicists. 
One advantage of working with complex threefolds is that we can reduce  our questions to the theory of complex surfaces by considering linear systems of divisors. 
Furthermore, for Calabi--Yau threefolds $X$, 
the second Chern class $c_{2}(X)$ and the K\"{a}hler cone $\mathcal{K}_{X}$ turn out to encode important information about $\mu$ (see \cite{Wil2, Wil4} for details).  
One purpose of this paper is to take the first step towards an investigation on 
how the Calabi--Yau structure affects the trilinear form $\mu$ and the Chern classes of the underlying manifold. \\

It is worth mentioning some relevant work from elsewhere. Let $(X,H)$ be a polarized Calabi--Yau threefold. 
A bound for the value $c_{2}(X)\cup H$ in terms of the triple intersection $H^{3}$ is well-known (see for example \cite{Wil3})  
and hence there are only finitely many possible Hilbert polynomials $\chi(X,\mathcal{O}_{X}(nH))=\frac{H^{3}}{6}n^{3} +\frac{c_{2}(X)\cup H}{12}n$ for such $(X,H)$.   
By the footnote below 
and standard Hilbert scheme theory, we know that the Calabi--Yau threefold $X$ belongs to a finite number of families. 
This implies that once we fix a positive integer $n \in \N$, there are only finitely many diffeomorphism classes of polarized Calabi--Yau threefolds $(X,H)$ with $H^{3}=n$,  
and in particular only finitely many possibilities for the Chern classes $c_{2}(X)$ and $c_{3}(X)$ of $X$.  
Explicit bounds on the Euler characteristic $c_{3}(X)$ in terms of $H^{3}$ for certain types of Calabi--Yau threefolds are given in \cite{BH, CK}; 
the idea of this article is to record the following simple explicit result which holds in general, and which may be  useful for both mathematicians and physicists. 
\begin{Thm} \label{MAIN} 
Let $(X,H)$ be a very amply polarized Calabi--Yau threefold, i.e. $x=H$ is a very ample divisor on $X$.  
Then the following inequality holds:
$$
-36\mu(x,x,x)-80\le \frac{c_{3}(X)}{2}=h^{1,1}(X)-h^{2,1}(X)\le 6\mu(x,x,x)+40.
$$
Moreover, the above inequality can be sharpened by replacing the left hand side by $-80$, $-180$ and right hand side by $28$, $54$ when $\mu(x,x,x)=1,3$ respectively 
\footnote{It is shown by K. Oguiso and T. Peternell \cite{OP} that we can always pass from an ample divisor $H$ on a Calabi--Yau threefold to a very ample one $10H$. }.  
\end{Thm}

In the last section,  we study the cubic form $\mu(x,x,x):H^{2}(X,\Z)\rightarrow \Z$ for a K\"{a}hler threefold $X$, assuming that $\mu(x,x,x)$ has a linear factor over $\R$.  
Some properties of the linear form and the residual quadratic form on $H^{2}(X,\R)$ are obtained;  
possible signatures of the residual quadratic form are determined under a certain condition (for example $X$ is a Calabi--Yau threefold).   

\section{Bound for $c_{2}(X)\cup H$} 

In this section, we collect some properties of the trilinear form and the second Chern classes of a Calabi--Yau threefold. 
We will always work over the field of complex numbers $\C$. \\

Let $X$ be a smooth K\"{a}hler threefold. 
Throughout this paper, we write $c_{i}(X)=c_{i}(TX)$ the $i$-th Chern class of the tangent bundle $TX$.  
K\"{a}hler classes constitute an open cone $\mathcal{K}_{X}\subset H^{1,1}(X,\C)\cap H^{2}(X,\R)$, called the  K\"{a}hler cone. 
The closure $\overline{\mathcal{K}_{X}}$  then consists of nef classes and hence is called the nef cone. 
The second Chern class $c_{2}(X) \in H^{4}(X, \Z)$ defines a linear function on $H^{2}(X,\R)$. 
Under the assumption that $X$ is minimal (for instance a Calabi--Yau threefold), 
results of Y. Miyaoka \cite{Miy} imply  that for any nef class $x \in \overline{\mathcal{K}_{X}}$, we have  $c_{2}(X)\cup x \ge 0$. \\

Let $X$ be a smooth complex threefold. 
We define a symmetric trilinear form $\mu:H^{2}(X,\Z)^{\otimes3} \rightarrow H^{6}(X,\Z)\cong \Z$ 
by setting $\mu(x,y,z)=x\cup y \cup z$ for $x,y,z \in H^{2}(X,\Z)$. 
By small abuse of notation we also use $\mu$ for its scalar extension.

\begin{Def}
A Calabi--Yau threefold $X$ is a complex projective smooth threefold with trivial canonical bundle $K_{X}$ such that $H^{1}(X, \mathcal{O}_{X})=0$. 
\end{Def}
For a Calabi--Yau threefold $X$, the exponential exact sequence gives an identification $\mathrm{Pic}(X)=H^{1}(X, \mathcal{O}_X ^{\times}) \cong H^{2}(X,\mathbb{Z})$. 
The divisor class $[D]$ is then identified with the first Chern class $c_{1}(\mathcal{O}_{X}(D))$ of the associated line bundle $\mathcal{O}_{X}(D)$. 
In the following we freely use this identification. \\

The Hirzebruch--Riemann--Roch theorem for a Calabi--Yau threefold $X$ states that $\chi(X,\mathcal{O}_{X}(D))=\frac{1}{6} \mu(x,x,x)+\frac{1}{12}c_{2}(X)\cup x$ for any $x=D \in H^{2}(X,\Z)$. 
Therefore 
$$
2\mu(x,x,x)+c_{2}(X)\cup x\equiv0 \pmod{12}. 
$$
In particular, $c_{2}(X)\cup x$ is an even integer for any $x \in H^{2}(X,\Z)$. 
In the case when the cohomology is torsion-free, this also follows from the fact $p_{1}(X)=-2c_{2}(X)$ and  Wall's Theorem \ref{Wall}. 
The role played by $p_{1}(X)$ in his theorem is replaced by $c_{2}(X)$ for Calabi--Yau threefolds. \\ 

For a compact complex surface $S$, the geometric genus $p_{g}(S)$ is defined by $p_{g}(S)=\dim_{\C}H^{0}(S,\Omega_{S}^{2})$.  
The basic strategy we take in the following is to reduce the question on Calabi--Yau threefolds to compact complex surface theory by considering linear systems of divisors. 

\begin{Prop} \label{NC} 
Let $X$ be a Calabi--Yau threefold. 
\begin{enumerate}
\item For any ample $x=H \in \mathcal{K}_{X}\cap H^{2}(X,\Z)$ with $|H|$ free and $\dim_{\C}|H|\ge 2$, the following inequalities hold. 
$$
\frac{1}{2}c_{2}(X)\cup x  \le 2\mu(x,x,x)+C
$$
where $C=18$ when $\mu(x,x,x)$ even and $C=15$ otherwise. 
\item If furthermore the canonical map $\Phi_{|K_{H}|}:H\rightarrow \mathbb{P}^{|K_{H}|}$ (which is given by the restriction of the map $\Phi_{|H|}$ to $H$) is birational onto its image, the following inequality holds. 
$$
\frac{1}{2}c_{2}(X)\cup x  \le \mu(x,x,x)+20 
$$
\item If furthermore the image of the canonical map in (2) is generically an intersection of quadrics, the following inequality holds.
$$
c_{2}(X)\cup x \le \mu(x,x,x)+48
$$
\end{enumerate}
\end{Prop}

\begin{proof}
(1) By Bertini's theorem, a general member of the complete linear system $|H|$ is irreducible and gives us a smooth compact complex surface $S \subset X$. 
Applying the Hirzebruch--Riemann--Roch theorem and the Kodaira vanishing theorem to the ample line bundle $\mathcal{O}_{X}(H)$, 
we can readily show that the geometric genus $p_{g}(S)=\frac{1}{6}\mu(x,x,x)+\frac{1}{12}c_{2}(X)\cup x-1$. 
Since $K_{S}$ is ample, the surface $S$ is a minimal surface of general type. 
Then the Noether's inequality $\frac{1}{2}K_{S}^{2}\ge p_{g}(S)-2$ yields the desired two equalities depending on the parity of $K_{S}^{2}=\mu(x,x,x)$.\\
(2) The proof is almost identical to the first case. 
Since the surface $S$ obtained above is a minimal canonical surface, i.e. the canonical map $\Phi_{|K_{S}|}:S\rightarrow \mathbb{P}^{|K_{S}|}$ is birational onto  its image,  
the Castelnuovo inequality for minimal canonical surfaces $K_{S}^{2}\ge 3p_{g}(S)-7$ yields the inequality.\\
(3) We say that an irreducible variety $S\subset \mathbb{P}^{p_{g}-1}$ is generically an intersection of quadrics 
if $S$ is one component of the intersection of all quadrics through $S$. 
In this case, M. Reid \cite{MR} improved the above inequality to $K_{S}^{2}\ge 4p_{g}(S)+q(S)-12$. 
The irregularity $q(S)=\dim_{\C}H^{1}(S,\mathcal{O}_{S})=0$ in our case. 
\end{proof}

If $x\in \mathcal{K}_{X}$ is very ample, the conditions in Proposition \ref{NC} (1) and (2) are automatically satisfied. 
The first two inequalities are optimal in the sense that equalities hold for the complete intersection Calabi--Yau threefolds $\mathbb{P}_{(1^{4},4)}\cap(8)$ and $\mathbb{P}^{4}\cap (5)$. \\

It is worth noting that polarized Calabi--Yau threefolds $(X,H)$ with $\Delta$-genus $\Delta(X, H)\le 2$ are classified by K. Oguiso \cite{O} 
and it is observed by the second author \cite{Wil3} that the inequality $c_{2}(X)\cup H \le 10 H^{3}$ holds for those with $\Delta(X,H)>2$. 
R. Schimmrigk's experimental observation \cite{Sc} however conjectures the existence of a better linear upper bound of $c_{2}(X)$  
for Calabi--Yau hypersurfaces in weighted projective spaces.

\begin{Prop}
The surface $S$ in the proof of Proposition \ref{NC} is a minimal surface of general type with non-positive second Segre class $s_{2}(S)$. 
$s_{2}(S)$ is negative if and only if $c_{2}(X)$ is not identically zero. 
\end{Prop}
\begin{proof}
Let $i:S \hookrightarrow X$ be the inclusion and we identify $H^{4}(S,\Z)\cong \Z$. 
A simple computation shows $c_{1}(S)=-i^{*}(x)$ and $c_{2}(S)=\mu(x,x,x)+c_{2}(X)\cup x$. 
Since $x \in \mathcal{K}_{X}$, $s_{2}(S)=c_{1}(S)^{2}-c_{2}(S)=-c_{2}(X)\cup x \le 0$ by the result of Y. Miyaoka \cite{Miy}. 
The second claim follows from the fact that $\mathcal{K}_{X} \subset H^{2}(X,\R)$ is an open cone.   
\end{proof}
If $X$ is a Calabi--Yau threefold and the linear form $c_{2}(X)$ is identically zero, it is well known that $X$ is the quotient of an Abelian threefold by a finite group acting freely on it.

\section{Bound for $c_{3}(X)$}
In this section, we apply to smooth projective threefolds the Fulton--Lazarsfeld  theory for nef vector bundles developed by J.P. Demailly, T. Peternell and M. Schneider \cite{DPS}.  
This gives us several inequalities among  Chern classes and cup products of certain cohomology classes. 
When $X$ is a Calabi--Yau threefold, these inequalities simplify and provide us with effective bounds for the  Chern classes.\\

Recall that a vector bundle $E$ on a complex manifold $X$ is called nef if the Serre line bundle $\mathcal{O}_{\mathbb{P}(E)}(1)$ on the projectivized bundle $\mathbb{P}(E)$ is nef. 

\begin{Thm}[J.P. Demailly, T. Peternell, M. Schneider \cite{DPS}] \label{DPS}
Let $E$ be a nef vector bundle over a complex manifold $X$ equipped with a K\"{a}hler class $\omega_{X} \in \mathcal{K}_{X}$. 
Then for any Schur polynomial $P_{\lambda}$ of degree $2r$ and any complex submanifold $Y$ of dimension $d$, we have $\int_{Y}P_{\lambda}(c(E))\wedge \omega_{X}^{d-r} \ge 0$.
\end{Thm}
Here we let $\deg c_{i}(E)=2i$ for $0\le i \le \rk E$ and the Schur polynomial $P_{\lambda}(c(E))$ of degree $2r$ is defined by $P_{\lambda}(c(E))=\det(c_{\lambda_{i}-i+j}(E))$ 
for each partition $\lambda=(\lambda_{1},\lambda_{2},\dots) \dashv r$ of a non-negative integer $r\le \dim Y$ with $\lambda_{k}\ge \lambda_{k+1}$ for all $k\in \N$.  

\begin{Ex} (\cite{Laz}, page 118) 
Let $X$ be a complex threefold and $E$ a vector bundle of $\rk E=3$, then
$$
P_{(1)}(c(E))=c_{1}(E), \ \ 
P_{(2)}(c(E))=c_{2}(E), \ \ 
P_{(1,1)}(c(E))=c_{1}(E)^{2}-c_{2}(E)
$$
$$
P_{(3)}(c(E))=c_{3}(E), \ \ \ 
P_{(2,1)}(c(E))=c_{1}(E)\cup c_{2}(E)-c_{3}(E),
$$
$$
P_{(1,1,1)}(c(E))=c_{1}(E)^{3}-2c_{1}(E)\cup c_{2}(E)+c_{3}(E). 
$$

\end{Ex}

\begin{Prop}
Let $X$ be a smooth projective threefold, $x,y \in \mathcal{K}_{X}\cap H^{2}(X,\Z)$ and assume $x$ is very ample, then the following inequalities hold. 
\begin{enumerate}
\item $8\mu(x,x,x)+2c_{2}(X)\cup x \ge 4\mu(c_{1}(X),x,x)+c_{3}(X)$
\item $64\mu(x,x,x)+4\mu(c_{1}(X),c_{1}(X),x)+4c_{2}(X)\cup x  +c_{3}(X)
           \\ \ \ \ge 32\mu(c_{1}(X),x,x)+c_{1}(X)\cup c_{2}(X)$
\item $80\mu(x,x,x)+10\mu(c_{1}(X),c_{1}(X),x)+2c_{1}(X)\cup c_{2}(X)
           \\ \ \ \ge 40\mu(c_{1}(X),x,x)+\mu(c_{1}(X),c_{1}(X),c_{1}(X))+10c_{2}(X)\cup x +c_{3}(X)$
\item $12\mu(x,x,y)+c_{2}(X)\cup y \ge 4\mu(c_{1}(X),x,y)$
\item $24 \mu(x,x,y) +\mu(c_{1}(X),c_{1}(X),y) \ge 8\mu(c_{1}(X),x,y)+c_{2}(X)\cup y$
\item $6\mu(x,y,y)\ge\mu(c_{1}(X),y,y)$
\end{enumerate}
\end{Prop}

\begin{proof}
The very ample divisor $x=H$ gives us an embedding $\Phi_{|H|}:X\rightarrow \mathbb{P}(V)$, where $V=H^{0}(X,\mathcal{O}_{X}(H))$.  
Using the Euler sequence and the Koszul complex, we obtain the following exact sequence of sheaves 
$$
0 \longrightarrow \Omega_{\mathbb{P}(V)}^{k+1} \longrightarrow \bigwedge^{k+1}V \otimes \mathcal{O}_{\mathbb{P}(V)}((-k-1)H) \longrightarrow  \Omega_{\mathbb{P}(V)}^{k}\longrightarrow 0 
$$
for each $1\le k \le \dim_{\C}V-1$. 
We see that $\Omega_{\mathbb{P}(V)}(2H)$ is a quotient of $\mathcal{O}_{\mathbb{P}(V)}^{\oplus\binom{\dim_{\C}V}{2}}$. 
The vector bundle $\Omega_{X}(2H)$ is then generated by global sections because it is a quotient of the globally generated vector bundle $\Omega_{\mathbb{P}(V)}|_{X}(2H)$. 
We hence conclude that $\Omega_{X}(2H)$ is a nef vector bundle. 
Applying Theorem \ref{DPS} (or rather the inequalities derived using the above example) to our nef vector bundle $\Omega_{X}(2H)$, straightforward computation shows the desired inequalities.  
\end{proof}

The above result (with appropriate modification) certainly carries over to complex manifolds of dimension other than $3$.

\begin{Cor} \label{FL} 
Let $X$ be a Calabi--Yau threefold, $x,y \in \mathcal{K}_{X}\cap H^{2}(X,\Z)$ and assume $x$ is very ample, then the following inequalities hold. 
\begin{enumerate}
\item $8\mu(x,x,x)+2c_{2}(X)\cup x \ge c_{3}(X)$ 
\item $64\mu(x,x,x)+4c_{2}(X)\cup x +c_{3}(X)\ge 0$
\item $80\mu(x,x,x)\ge 10c_{2}(X)\cup x+c_{3}(X)$
\item $24 \mu(x,x,y)\ge c_{2}(X)\cup y$
\end{enumerate}
\end{Cor}

In recent literature there has been some interest in finding practical bounds for topological invariants of Calabi--Yau threefolds.  
As is mentioned in the introduction, the standard Hilbert scheme theory assures 
that possible Chern classes of a polarized Calabi--Yau threefold $(X,H)$ are in principle bounded once we fix a triple intersection number $H^{3}=n \in \N$, 
but now that we have effective bounds for the Chern classes (with a bit of extra data for the second Chern class $c_{2}(X)$) as follows.  
Recall first that it is shown by K. Oguiso and T. Peternell \cite{OP} that we can always pass from an ample divisor $H$ on a Calabi--Yau threefold to a very ample one $10H$.   
Then the last inequality in Corollary \ref{FL} says that once we know the trilinear form $\mu$ on the ample cone $\mathcal{K}_{X}$ 
there are only finitely many possibilities for the linear function $c_{2}(X) : H^{2}(X,\Z) \rightarrow \Z$. 
We shall now give a simple explicit formula to give a range of the Euler characteristic $c_{3}(X)$ of a Calabi--Yau threefold $X$. 

\begin{Thm} \label{MAIN} 
Let $(X,H)$ be a very amply polarized Calabi--Yau threefold, i.e. $x=H$ is a very ample divisor on $X$.  
Then the following inequality holds:
$$
-36\mu(x,x,x)-80\le \frac{c_{3}(X)}{2}=h^{1,1}(X)-h^{2,1}(X)\le 6\mu(x,x,x)+40.
$$
Moreover, the above inequality can be sharpened by replacing the left hand side by $-80$, $-180$ and right hand side by $28$, $54$ when $\mu(x,x,x)=1,3$ respectively. 
\end{Thm}
\begin{proof}
This is readily proved by combining Proposition \ref{NC} (1), (2) and Corollary \ref{FL} (1), (2), (4).  
\end{proof}

The smallest and largest known Euler characteristics $c_{3}(X)$ of a Calabi--Yau threefold $X$ are $-960$ and $960$ respectively. 
Our formula may replace the question of finding a range of $c_{3}(X)$ by that of estimating the value $\mu(x,x,x)$ for an ample class $x \in \mathcal{K}_{X} \cap H^{2}(X,\Z)$.

\section{Quadratic forms associated with special cubic forms}

In this section we further study the cubic form $\mu(x,x,x):H^{2}(X,\Z)\rightarrow \Z$ for a K\"{a}hler threefold $X$, assuming that $\mu(x,x,x)$ has a linear factor over $\R$.  
We will see that the linear factor and the residual quadratic form are not independent. 
Possible signatures of the residual quadratic form are also determined under a certain condition. 
If the second Betti number $b_{2}(X)>3$, the residual quadratic form may endow the second cohomology $H^{2}(X,\Z)$ mod torsion with a lattice structure. \\

We start with fixing our notation. 
Let $\xi:V\rightarrow \R$ be a real quadratic form. 
Once we fix a basis of the $\R$-vector space $V$, $\xi$ may be represented as $\xi(x)=x^{t}A_{\xi}x$ for some symmetric matrix $A_{\xi}$. 
The signature of a quadratic form $\xi$ is a triple $(s_{+},s_{0},s_{-})$ 
where $s_{0}$ is the number of zero eigenvalues of $A_{\xi}$ and $s_{+}$ $(s_{-})$ is the number of positive (negative) eigenvalues of $A_{\xi}$. 
$A_{\xi}$ also defines a linear map $A_{\xi}:V\rightarrow V^{\vee}$ (or a symmetric bilinear form $A_{\xi}:V^{\otimes2}\rightarrow \R$).  
The quadratic form $\xi$ is called (non-)degenerate if $\dim_{\R} Ker(A_{\xi})>0 \ (=0)$. 
We say that $\xi$ is definite if it is non-degenerate and either $s_{+}$ or $s_{-}$ is zero, and indefinite otherwise. \\

Let $X$ be a K\"{a}hler threefold and assume that its cubic form $\mu(x,x,x)$ factors as $\mu(x,x,x)=\nu(x)\xi(x)$, where $\nu$ is linear and $\xi$ is quadratic map $H^{2}(X,\R) \rightarrow \R$. 
We can always choose the linear form $\nu$ so that it is positive on the K\"{a}hler cone $\mathcal{K}_{X}$. 
It is proven (see the proof of Lemma 4.3 in \cite{Wil}) that there exists a non-zero point on the quadric $Q_{\xi}=\{x\in H^{2}(X,\R) \ | \ \xi(x)=0\}$ and hence $\xi$ is indefinite 
provided that the irregularity $q(X)=\dim_{\C}H^{1}(X,\mathcal{O}_{X})=0$ and the second Betti number $b_{2}(X)>3$.  

\begin{Prop}
Let $X$ be a K\"{a}hler threefold. 
Assume that the trilinear form $\mu(x,x,x)$ decomposes as $\nu(x)\xi(x)$ over $\R$ (if the quadratic form is not a product of linear forms, then we may work 
over $\Q$)
and the linear form $\nu$ is positive on the K\"{a}hler cone $\mathcal{K}_{X}$. 
Then the following hold. 
\begin{enumerate}
\item $\dim_{\R} Ker(A_{\xi})\le 1$.   
If $\xi$ is a degenerate quadratic form, its restriction $\xi|_{H_{\nu}}$ to the hyperplane $H_{\nu}=\{x \in H^{2}(X,\R)\ | \ \nu(x)=0\}$ is non-degenerate. 
\item If the irregularity $q(X)=0$ (for example a Calabi--Yau threefold), then the signature of $\xi$ is either $(2,0,b_{2}(X)-2)$, $(1,1,b_{2}(X)-2)$ or $(1,0,b_{2}(X)-1)$. 
\item The above three signatures are realized by some Calabi--Yau threefolds with $b_{2}(X)=2$.  
\end{enumerate}
\end{Prop}

\begin{proof}
(1) Let $\omega_{X} \in \mathcal{K}_{X}$ be a K\"{a}hler class. 
The Hard Lefschetz theorem states that the map $H^{2}(X,\R)\rightarrow H^{4}(X,\R)$ defined by $\alpha \mapsto \omega_{X} \cup \alpha$ is an isomorphism.
Hence the cubic form $\mu(x,x,x)$ depends on exactly $b_{2}(X)$ variables. 
Then the quadratic form $\xi$ must depend on at least $b_{2}(X)-1$ variables and thus we have $\dim_{\R}(Ker(A_{\xi}))\le 1$. 
Assume next that the quadratic form $\xi$ is degenerate. 
Then the linear form $\nu$ is not the zero form on $Ker(A_{\xi})$ (otherwise $\mu(x,x,x)$ depends on less than $b_{2}(X)$ variables).  
The restriction $\xi|_{H_{\nu}}$ is non-degenerate because $H^{2}(X,\R)=H_{\nu} \oplus Ker(A_{\xi})$ as a $\R$-vector space. \\
(2) Let $L_{1} \in \mathcal{K}_{X} \cap H^{2}(X,\R)$ be an ample class such that $\mu(L_{1},L_{1},L_{1})=1$.  
Since the K\"{a}hler cone $\mathcal{K}_{X}\subset H^{2}(X,\R)$ is an open cone, $X$ is projective by the Kodaira embedding theorem. 
Then the Hodge index theorem states 
that the symmetric bilinear form $b_{\mu,L_{1}}=\mu(L_{1},*,**):H^{2}(X,\R)^{\otimes 2}\cong (NS(X)\otimes \R)^{\otimes 2} \rightarrow \R$ has signature $(1,0,b_{2}(X)-1)$, 
where $NS(X)$ is the Neron--Severi group of $X$. 
Note that $\dim_{\R} (L_{1} ^{\perp} \cap H_{\nu}) \ge b_{2}(X) -2$, 
where $L_1^{\perp}$ denotes the orthogonal space to $L_{1}$ with respect to the non-degenerate bilinear form $b_{\mu , L_1}$.   
We then have two cases; the first is when $\dim_{\R} (L_{1} ^{\perp} \cap H_{\nu})=b_2(X) -1$ (i.e. $L_{1}^{\perp} = H_{\nu}$). 
In this case we can write down a basis $L_{2}, \ldots , L_{b_{2}(X)}$ for the subspace $H_{\nu}$ which diagonalizes the quadratic form $b_{\mu , L_{1}}|_{H_{\nu}}$,  
and hence (noting that $L_{1} \not\in H_{\nu})$ the Gramian matrix of $b_{\mu , L_{1}}$ 
with respect to the basis $L_{1}, \ldots , L_{b_{2}(X)}$ of $H^2(X, \R)$ is $A_{b_{\mu,L_{1}}}=(b_{\mu,L_{1}}(L_{i},L_{j}))=diag(1,-1,\dots,-1)$. 
If $\dim_{\R} (L_{1} ^{\perp} \cap H_{\nu})=b_{2}(X)-2$, then we can write down a basis $L_{2} , \ldots, L_{b_{2}(X)-1}$ 
for the subspace $L_{1} ^{\perp} \cap H_{\nu}$ which diagonalizes the quadratic form $b_{\mu , L_{1}}|_{L_{1} ^{\perp} \cap H_{\nu}}$, 
and then extend it to a basis $L_2, \ldots , L_{b_2(X)}$ of $H_\nu$.  
Thus in both cases $L_1 , \ldots ,L_{b_2(X)}$ is a basis for $H^2 (X, \R)$; 
the corresponding matrix $A_{b_\mu ,L_1}$ will not be diagonal in this second case, 
but the first $(b_{2}(X) -1)$-principal minor is, with one $+1$ and $b_{2}(X)-2$ entries $-1$ on the diagonal.

Let us define a new basis $\{M_{i}\}_{i=1}^{b_{2}(X)}$ of $H^{2}(X,\R)$ by setting $M_{i}=L_{i}$ for $1\le i \le b_{2}(X)-1$ and  
$$
M_{b_{2}(X)}=L_{b_{2}(X)}+\sum_{i=2}^{b_{2}(X)}b_{\mu,L_{1}}(L_{i},L_{b_{2}(X)})L_{i} \in H_{\nu}. 
$$ 
Let $x=\sum_{i=1}^{b_2{(X)}} a_{i}M_{i}$. 
Then the hyperplane $H_{\nu}$ is defined by the equation $a_{1}=0$ and the K\"{a}hler cone $\mathcal{K}_{X}$ lies on the side where $a_{1}>0$ by the assumption on $\nu$.  
Therefore we have 
$$
\mu(x,x,x)=a_{1}(a_{1}^{2}-\sum_{i=2}^{b_{2}(X)-1}a_{i}^{2}+Ca_{1}a_{b_{2}(X)}+Da_{b_{2}(X)}^{2})
$$
for some (explicit) constants $C,D \in \R$.
Since the quadratic form is positive on the the K\"{a}hler cone $\mathcal{K}_{X}$, there must be at least one positive eigenvalue 
and hence possible signatures are $(2,0,b_{2}(X)-2)$, $(1,1,b_{2}(X)-2)$ and $(1,0,b_{2}(X)-1)$. \\
(3) Consider a Calabi--Yau threefold $X_{7}^{II}(1,1,1,2,2)^{2}_{-186}$ from page 575 \cite{HLY} 
given as a resolution of a degree $7$ hypersurface in the weighted projective space $\mathbb{P}_{(1,1,1,2,2)}$. 
Its cubic form is given by $a_{1}(14a_{1}^{2} +21a_{1}a_{2}+9a_{2}^{2})$, whose quadratic form has signature $(2,0,0)$. 
The cubic form of a hypersurface Calabi--Yau threefold $(\mathbb{P}^{3}\times \mathbb{P}^{1})\cap(4,2)$ is $2a_{1}^{3}+12a_{1}^{2}a_{2}$,  
whose quadratic form has signature either $(1,0,1)$ or $(1,1,0)$, depending on its decomposition. 
\end{proof}

The restriction $\xi|_{H_{\nu}}$ may be degenerate if $\xi$ is non-degenerate. 
The cubic form of the above Calabi--Yau threefold $(\mathbb{P}^{3}\times \mathbb{P}^{1})\cap(4,2)$ gives an example of such phenomenon. 
Let $\nu(a)=2a_{1}$ and $\xi(a)=a_{1}(a_{1}+6a_{2})$. 
Then $\xi$ is hyperbolic and non-degenerate, but its restriction to $H_{\nu}$ is trivial. \\

Let $X$ be a K\"{a}hler threefold. 
If $b_{2}(X)>3$, the cubic form $\mu$ cannot consist of three linear factors over $\R$ and hence if $\mu$ contains a linear factor it must be rational (see also the comment after Lemma 4.2 \cite{Wil}).  
Hence an appropriate scalar multiple of $\xi$ endows the second cohomology $H^{2}(X,\Z)$ mod torsion with a lattice structure. 

\begin{Ex}
Let $X$ be a K3 surface and $\iota_{S}$ an involution of $S$ without fixed point (hence the quotient $S/\langle \iota_{S}\rangle$ is an Enriques surface).  
Let $E$ be an elliptic curve with the canonical involution $\iota_{E}$. 
Then we can define a new involution $\iota$ of $S\times E$ given by  $\iota=(\iota_{S},\iota_{E})$. 
The quotient $X=(S\times E)/\langle \iota \rangle$ is a Calabi--Yau threefold with $b_{2}(X)=11$. 
The cubic form $\mu(x,x,x$) of $X$ has a linear factor (which, we assume, is positive on the  K\"{a}hler cone $\mathcal{K}_{X}$) and the residual quadratic form $\xi$ has signature $(1,1,9)$. 
More precisely, the lattice structure on $H^{2}(X,\Z)$ mod torsion associated with appropriate $\xi$ is given by $U\oplus E_{8}(-1)\oplus \langle 0\rangle $, 
where $U$ is the hyperbolic lattice, $E_{8}(-1)$ is the root lattice of type $E_{8}$ multiplied by $-1$ and $\langle 0\rangle$ is a trivial lattice of rank $1$.  
\end{Ex}

\begin{Prop}
Let $G$ be a finite group acting on a K\"{a}hler threefold $X$ and $\phi:G\rightarrow GL(H^{2}(X,\Z))$ the induced representation.  
Assume that the trilinear form decomposes $\mu(x,x,x)=\nu(x)\xi(x)$ as above. 
Then the image of $\phi:G\rightarrow GL(H^{2}(X,\Z))$ lies in the orthogonal group $O(\xi)$ associated with the quadratic form $\xi$.   
\end{Prop}
\begin{proof}
Since the cubic form $\mu:H^{2}(X,\R)\rightarrow \R$ is invariant under $G$, it is enough to show that the linear form $\nu$ is invariant under $G$. 
There exists $x\in \mathcal{K}_{X}$ such that $\R x$ is a trivial representation of $G$ (by averaging a K\"{a}hler class over $G$) 
and then the representation $\phi$ is a direct sum of two subrepresentations $\R x\oplus H_{\nu}$. 
Since $\nu$ is a linear form, this shows the invariance of $\nu$ under $G$.  
\end{proof}

This proposition may be useful to study group actions on the cohomology group $H^{2}(X,\Z)$.

\subsection*{Acknowledgement}
The present work was initiated during the the first author's stay at the Workshop on Arithmetic and Geometry of K3 surfaces and Calabi--Yau threefolds, August 16-25, 2011. 
He is grateful to the organizers of the workshop and Fields Institute for their warm hospitality and partial travel support.

\par\noindent{\scshape \small
Department of Mathematics, University of British Columbia \\
51984 Mathematics Rd , Vancouver, BC, V6T 1Z2, CANADA.}
\par\noindent{\ttfamily kanazawa@math.ubc.ca}
\break
\par\noindent{\scshape \small
Department of Pure Mathematics, University of Cambridge \\
16 Wilberforce Road, Cambridge CB3 1SB, UK.}
\par\noindent{\ttfamily pmhw@dpmms.cam.ac.uk}

\end{document}